\documentclass[a4paper]{amsart}
\usepackage{amsfonts}
\usepackage{amsmath,amsthm,amssymb}
\usepackage[abs]{overpic}
\usepackage{graphicx}
\usepackage{color}
\usepackage{comment}

\newtheorem{theorem}{Theorem}[section]
\newtheorem{lemma}[theorem]{Lemma}
\newtheorem{proposition}[theorem]{Proposition}

\theoremstyle{definition}
\newtheorem{definition}[theorem]{Definition}
\theoremstyle{remark}
\newtheorem{remark}[theorem]{Remark}
\newtheorem{example}[theorem]{Example}

\numberwithin{equation}{section}

\newcommand{\e}{\varepsilon}
\newcommand{\de}{\delta}
\newcommand{\omu}{\overline{\mu}}
\makeatletter

\@addtoreset{figure}{section}
\makeatother
\begin{document} 

\title{Milnor invariants of covering links}
\author{Natsuka Kobayashi, Kodai Wada and Akira Yasuhara}

\address{Tokyo Metropolitan Fujimori High School, Nagabusa, Hachioji, Tokyo 193-0824, Japan}
\email{ntk08g@gmail.com}

\address{Department of Mathematics, School of Education, Waseda University, Nishi-Waseda 1-6-1, Shinjuku-ku, Tokyo, 169-8050, Japan}
\email{k.wada@akane.waseda.jp}

\address{Department of Mathematics, Tokyo Gakugei University, 4-1-1 Nukuikita-machi, Koganei-shi, Tokyo, 184-8501, Japan}
\email{yasuhara@u-gakugei.ac.jp}

\thanks{ The third author is partially supported by a Grant-in-Aid for Scientific Research (C) 
($\#$23540074) of the Japan Society for the Promotion of Science.}


\begin{abstract} 
We consider Milnor invariants for certain covering links as a generalization of covering linkage invariants 
formulated by R. Hartley and K.~Murasugi.
A set of Milnor invariants for covering links is a cobordism invariant of a link, 
and that  this invariant can distinguish some links 
for which the ordinary 
Milnor invariants coincide.  
Moreover, for a Brunnian link $L$, the first non-vanishing  
Milnor invariants of $L$ is modulo-$2$ congruent to a sum of  
Milnor invariants of covering links.
As a consequence, a sum of linking numbers of \lq iterated\rq~covering links gives the first non-vanishing Milnor invariant of $L$ modulo $2$.
\end{abstract}

\maketitle

\section{Introduction}
For a link $L$ in the $3$-sphere $S^{3}$, we consider a branched cover of $S^{3}$ 
branched over components of $L$. 
The set of linking numbers between  2-component sublinks of the preimage  $\widetilde{L}$ of $L$ 
had been recognized as a useful invariant for knots and links, see for example~\cite{BS}, \cite{L}, \cite{P} and \cite{R}.
R. Hartley and K.~Murasugi~\cite{HM} called this invariant the {\em covering linkage invariant}. 

J. Milnor~\cite{M}, \cite{M2} defined a family of invariants for a link 
indexed by sequences of integers in $\{1,2,\ldots,n\}$, where $n$ is the number of components of the link.
Since the Milnor invariant of a link for a length two sequence $ij$ coincides with the linking number of the $i$th and $j$th components of the link,
we could regard the Milnor invariants for sequences with the length at least 3 as a kind of higher order linking numbers.

So it seems to be natural to consider the Milnor invariants for (sublinks of) $\widetilde{L}$ as a generalization of covering 
linkage invariants. 
In fact, for a prime number $p$, T.D. Cochran and K.E. Orr~\cite{CO} defined \lq mod $p$ or $p$-adic versions\rq~
Milnor invariants for links in the $\mathbb{Z}_p$-homology 3-sphere. 
Their Milnor invariants can be also defined for 
$\widetilde{L}$ in the $p$-fold cyclic branched cover  
of $S^3$ branched over a component of $L$.  

In order to investigate the ordinary Milnor invariants for $\widetilde{L}$, we only consider a simple case as follows:
Let $L$ be a link with a trivial component $K$ such that the linking numbers between $K$ 
and the other components are even.  
For the double branched cover $\Sigma(K)$ branched over $K$, 
a link in $\Sigma(K)$ which consists of components of the preimage of each component of $L\setminus K$ 
is said to be a {\em covering link}.
Note that there are $2^{n-1}$ covering links in $\Sigma(K)$, where $n$ is the number of components of $L$.   
 (We remark that in \cite{CO}, they call the preimage $\widetilde{L}$ of $L$ \lq the\rq~ covering link.) 
 Since $\Sigma(K)$ is also $S^3$, we can define the ordinary Milnor invariants for 
 each covering link.

It is known that the Milnor invariants are cobordism invariants~\cite{C}.  
It is also true that a set of the  Milnor invariants for certain covering 
links are cobordism invariants of $L$ (Theorem~\ref{cobordism}). 
In~\cite{CO}, Cochran and Orr show that their invariants are ($p$-)cobordism invariants of
the {\em covering} link $\widetilde{L}$.

In~\cite{CO}, they do not make it clear if 
their Milnor invariants of the covering link are useful 
when the ordinary Milnor invariants are useless. 
Our invariants, the Milnor invariants of a covering link, can distinguish some links for which 
the ordinary Milnor invariants coincide, 
see Remark~\ref{remstrong}.

For a Brunnian link $L$ in $S^{3}$,
we show that 
the first non-vanishing  
Milnor invariants of $L$ is modulo-$2$ congruent to a sum of  
Milnor invariants of certain covering links (Theorem~\ref{thmod2}).
In~\cite{Murasugi}, Murasugi proved that Milnor invariants of a link are linking numbers in appropriate nilpotent covering spaces of $S^{3}$ branched over the link,
i.e., he described the exact correspondence between Milnor invariants and covering linkage invariants.
While it looks that our result is weaker than the result of Murasugi, 
the authors believe that the result is worth mentioning because  
it is hard to treat nilpotent covers in general.

\section{Milnor invariants}
\label{Milnor inv}
Let $L$ be an oriented ordered $n$-component link in $S^{3}$.
Let $p$ be $0$ or a prime number. 
Let $G$ be the fundamental group of the complement of $L$ and
$G_{q}^{p}$ a normal subgroup of $G$ generated by $[x,y]z^{p}$ for $x\in G,~y,z\in G_{q-1}^{p}$,
where $[x,y]$ is the commutator of $x$ and $y$,
and $G_{1}^{p}$ means $G$. 
(We remark that in \cite{S}, $G$ is defined to be $G^p_0$.) 

It is shown by similar to Theorem 4 in~\cite{M2} that 
the quotient group $G/G_{q}^{p}$ is isomorphic to a group with the following presentation:
\[
\langle \alpha_{1},\alpha_{2},\ldots,\alpha_{n}~\vline~[\alpha_{i},\lambda^q_{i}] (i=1,2,\ldots,n), F_{q}^{p}\rangle,
\]
where $\alpha_{i}$, $\lambda^q_{i}$ represent $i$th meridian and longitude of $L$ respectively and 
$F$ is a free group generated by $\alpha_{1},\alpha_{2},\ldots,\alpha_{n}$.
In particular, $\lambda^q_{i}$ can be chosen as a word in $\alpha_{1},\alpha_{2},\ldots,\alpha_{n}$
so that $\lambda^q_{i}=\lambda^0_i$.

We introduce the {\it Magnus $\mathbb{Z}_{p}$-expansion} of $\lambda_{j}^{q}$.
The Magnus $\mathbb{Z}_{p}$-expansion $E^{p}$ is an embedding homomorphism of $F$ to the formal power series ring in non-commutative variables $X_{1},X_{2},\ldots,X_{n}$ with $\mathbb{Z}_{p}$ coefficients defined by 
$E(\alpha_{i})=1+X_{i}$ and $E(\alpha_{i}^{-1})=1-X_{i}+X_{i}^{2}-X_{i}^{3}+\cdots~(i=1,2,\ldots,n)$~\cite[6.1 Lemma]{S}.
Then $E^{p}(\lambda_{j}^{q})$ can be written in the form
\[
E^{p}(\lambda_{j}^{q})=1+\sum_{k<q}\mu^{q}_L(i_{1}i_{2}\cdots i_{k}j)_{p}X_{i_{1}}X_{i_{2}}\cdots X_{i_{k}}
+\text{(terms of degree $\geq q$)},
\]
where a coefficient $\mu^{q}_L(i_{1}i_{2}\cdots i_{k}j)_{p}$ is defined for each sequence $i_{1}i_{2}\cdots i_{k}j$ of integers in $\{1,2,\ldots,n\}$.
Let $1+f$ be the Magnus $\mathbb{Z}_{p}$-expansion of an element of $F_{q}^{p}$.
Then the degree of any term in 
$f$ is at least $q$ (see~\cite[6.2 Lemma]{S}).
This implies that for $q<q'$
\[
\mu^{q}_L(i_{1}i_{2}\cdots i_{k}j)_{p}=\mu^{q'}_L(i_{1}i_{2}\cdots i_{k}j)_{p}.
\]
Taking $q$ sufficiently large,
we may ignore $q$ 
and hence denote $\mu^{q}_L(i_{1}i_{2}\cdots i_{k}j)_{p}$ by $\mu_L(i_{1}i_{2}\cdots i_{k}j)_{p}$.

For a sequence $I=i_{1}i_{2}\cdots i_{k}j~(k<q)$,
let $\Delta_{L}(I)_{p}$ be the ideal of $\mathbb{Z}_p$ generated by  $\mu(J)_{p}$'s, where $J$ is obtained from a proper subsequence of $I$ by permuting cyclicly.
Then the {\it Milnor invariant $\omu_{L}(I)_{p}$} is defined by 
\[
\omu_{L}(I)_{p}\equiv\mu_L(I)_{p}\mod{\Delta_{L}(I)_{p}}.
\]
The {\it length} of $\omu_{L}(I)_{p}$ means the length of $I$.

\begin{remark}\label{rem}
(1)~The ordinary Milnor invariant $\omu_L(I)$ given in ~\cite{M}, \cite{M2} is 
equal to $\omu_L(I)_0$. 
(In~\cite{M2}, $\Delta_{L}(I)_{0}$ is defined as the greatest common divisor of $\mu_{L}(J)_{0}$'s.)
When $p$ is a prime number, $\Delta_L(I)_p$ is equal to either $\{0\}$ or $\mathbb{Z}_p$ since 
$\mathbb{Z}_p$ is a field. 
Hence we essentially consider the first non-vanishing of $\omu_{L}(I)_{p}$ 
when $p$ is prime. 
Since $G_q^0\subset G_q^p$ for any prime number $p$, 
$\mu_L(I)_p$ is the modulo-$p$ reduction of $\mu_L(I)_0$. 
Taking $p$ sufficiently large, we may regard that 
$\omu_L(I)_p=\omu_L(I)_0$
if $\Delta_{L}(I)_{0}=\{0\}$. 

(2)~In order to state our results, we only need the original definition by Milnor \cite{M}, \cite{M2}. 
The definition in this section will be needed to apply a theorem of Stallings \cite[5.1 Theorem]{S} 
in the proof of Theorem~\ref{cobordism}.
\end{remark}

\section{Covering Milnor invariants}
Let $L=K_1\cup K_{2}\cup\cdots\cup K_{n+1}$ be an oriented $(n+1)$-component link in $S^3$
with $K_{n+1}$ is trivial and linking numbers of $K_{n+1}$ and $K_{i}$ are even for all $i=1,2,\ldots,n$.
Let $\Sigma(K_{n+1})$ be the double branched cover of $S^3$ branched over $K_{n+1}$ and $K^{\e}_{i}~(\subset \Sigma(K_{n+1}))$ a component of the preimage of $K_i$ $(\e\in\{0,1\}, i=1,2,\ldots,n)$.
Then we denote by $L(\e_1\e_2\cdots\e_n)$ a link $K^{\e_1}_1\cup K^{\e_2}_2\cup\cdots\cup K^{\e_n}_n~(\e_i\in\{0,1\})$ in $\Sigma(K_{n+1})$ 
and call it a {\em covering link} of $L$.
Since by the covering translation of $\Sigma(K_{n+1})$ there is an bijection from 
$\{L(0\e_2\cdots\e_n)~|~\e_i\in\{0,1\}\}$ to $\{L(1\de_2\cdots\de_n)~|~\de_i\in\{0,1\}\}$, 
we only consider the set $\{L(0\e_2\cdots\e_n)~|~\e_i\in\{0,1\}\}$ in this section 
(and the last section).
It is obvious that the set $\{L(0\e_2\cdots\e_n)~|~\e_i\in\{0,1\}\}$ is an invariant of $L$.
Note that $\Sigma(K_{n+1})$ is still $S^{3}$ because the branch set $K_{n+1}$ is trivial.
Hence Milnor invariants are naturally defined for covering links.
In particular for a sequence $I$ of $\{1,2,\ldots,n\}$, $M_{L}(I)_p=\{\omu_{L(0\e_2\cdots\e_n)}(I)_p~|~\e_i\in\{0,1\}\}$ is an invariant of $L$, 
where each element $i$ of $\{2,\ldots,n\}$ corresponds to the subindex of $K_{i}^{\e_{i}}$.
We call $M_{L}(I)_p$ the {\it covering $\mathbb{Z}_p$-Milnor invariant} of $L$. 
In particular, we denote $M_{L}(I)_0$ by $M_{L}(I)$ and call it the {\it covering Milnor invariant} of $L$. 
It is easy to see that $M_L(I)_p$ is an ambient isotopy invariant for each $p$ and 
$M_L(I)$ is the strongest invariant of them. 
But it is not clear whether each $M_L(I)_p$
is a cobordism invariant. 
Here, two $n$-component links $L=K_{1}\cup K_{2}\cup\cdots\cup K_{n}$ and $L'=K'_{1}\cup K'_{2}\cup\cdots\cup K'_{n}$ in $S^{3}$ are {\it cobordant} if there is a disjoint union of annuli $A_{1},A_{2},\cdots,A_{n}$ in $S^{3}\times [0,1]$ such that the boundary $\partial A_{i}=K_{i}\cup (-K'_{i})$ for each $i(=1,2,\ldots,n)$, 
where $-K'_{i}$ is the $i$th component $K'_{i}$ with the opposite orientation.
It is known that Milnor invariants of links are cobordism invariants~\cite{C}.
The same result holds for the first non-vanishing covering Milnor invariants as follows.

\begin{theorem}
\label{cobordism}
Let $L=K_{1}\cup K_{2}\cup\cdots\cup K_{n+1}$ be an oriented $(n+1)$-component link in $S^3$ 
with $K_{n+1}$ is trivial and linking numbers of $K_{n+1}$ and $K_{i}$ are even for $i=1,2,\ldots,n$.
For a sequence $I$, if $\Delta_{L(0\e_2\cdots\e_n)}(I)_0=\{0\}$  
for all $\e_{i}$ $(i\in\{2,\ldots,n\}, \e_{i}\in\{0,1\})$, 
then $M_{L}(I)$ is a cobordism invariant of $L$. That is, the first non-vanishing covering Milnor invariants are 
cobordism invariants. 
\end{theorem}

Let us first prove the following.

\begin{lemma}
\label{4mfdhom}
Let $W$ be a connected $4$-manifold with the first and second betti numbers are zero.
There is a prime number $p$ such that 
the first and second homology groups $H_{1}(W;\mathbb{Z}_{p})$, $H_{2}(W;\mathbb{Z}_{p})$ of $W$ with $\mathbb{Z}_{p}$ coefficients are trivial. Moreover, for any prime number $p'$ with $p'>p$, 
$H_{1}(W;\mathbb{Z}_{p'})\cong H_{2}(W;\mathbb{Z}_{p'})\cong\{0\}$.
\end{lemma}

\begin{proof}
Since the first and second betti numbers are zero by the hypothesis, 
there is 
a prime number $p$ 
such that
\[
\begin{array}{l}
H_{1}(W;\mathbb{Z})\otimes\mathbb{Z}_{p}
\cong H_{2}(W;\mathbb{Z})\otimes\mathbb{Z}_{p}
\cong \{0\}.
\end{array}
\]
Note that this is also true for any prime number which is greater than $p$. 

Since $W$ is connected,
$H_{0}(W; \mathbb{Z})\cong\mathbb{Z}$.
By the universal coefficient theorem for homology,
we have the following two short exact sequences:
\[
\begin{array}{l}
\{0\}\longrightarrow H_{1}(W;\mathbb{Z})\otimes\mathbb{Z}_{p}
\longrightarrow H_{1}(W;\mathbb{Z}_{p})
\longrightarrow Tor(\mathbb{Z},\mathbb{Z}_{p})
\longrightarrow \{0\},\\

\{0\}\longrightarrow H_{2}(W;\mathbb{Z})\otimes\mathbb{Z}_{p}
\longrightarrow H_{2}(W;\mathbb{Z}_{p})
\longrightarrow Tor(H_{1}(W;\mathbb{Z}),\mathbb{Z}_{p})
\longrightarrow \{0\},
\end{array}
\]
where $Tor(G_{1},G_{2})$ is the torsion product of abelian groups $G_{1}$ and $G_{2}$.
Now we have $Tor(\mathbb{Z},\mathbb{Z}_{p})\cong Tor(H_{1}(W;\mathbb{Z}),\mathbb{Z}_{p})\cong\{0\}$.
This completes the proof.
\end{proof}

\begin{proof}[Proof of Theorem~{\rm \ref{cobordism}}]
Let $L=K_1\cup K_{2}\cup\cdots\cup K_{n+1}$ (resp. $L'=K'_{1}\cup K'_{2}\cdots\cup K'_{n+1}$) be an oriented $(n+1)$-component link in $S^3$
with $K_{n+1}$ (resp. $K'_{n+1}$) is trivial and linking numbers of $K_{n+1}$ and $K_{i}$ (resp. $K'_{n+1}$ and $K'_{i}$) are even for $i=1,2,\ldots,n$.
Suppose that $L$ and $L'$ are cobordant,
i.e., there is a disjoint union of annuli $A_{1},A_{2},\ldots,A_{n+1}$ in $S^{3}\times [0,1]$ such that each boundary $\partial A_{i}=K_{i}\cup (-K'_{i})$ for $i=1,2,\ldots,n+1$.

Let $W$ be the double branched cover of $S^{3}\times [0,1]$ branched over $A_{n+1}$.
Fix $\e_{2},\ldots,\e_{n}$, we may assume that the covering links $L(0\e_{2}\cdots\e_{n})$ and $L'(0\e_{2}\cdots\e_{n})$
bound a disjoint union $\widetilde{\mathcal{A}}$ of annuli in $W$ that are components of the preimage of $A_{1}\cup A_{2}\cup\cdots\cup A_{n}$.
We note that $\partial W=\Sigma(K_{n+1})\cup(-\Sigma(K'_{n+1}))=S^{3}\cup(-S^{3})$. 
We may assume that $K_{n+1}$ is in the boundary of the 4-ball $B^4$ and bounds a properly embedded 2-disk in $B^4$ 
which is obtained from a disk in $\partial B^4$ bounded by $K_{n+1}$ by pushing it into $B^4$. 
 Let $\Sigma$ be the double branched cover of $B^4$ branched over the properly embedded disk. 
Let $N(\widetilde{\mathcal{A}})$ be a regular neighborhood of $\widetilde{\mathcal{A}}$, 
and $D=\Sigma\cup N(\widetilde{\mathcal{A}})$.
Let $E$ be the closure of  $W\setminus N(\widetilde{\mathcal{A}})$.
Then we have $D\cup E=W\cup \Sigma$
and note that  $D\cap E$ is homeomorphic to 
$\Sigma(K_{n+1})\setminus N(L(0\e_{2}\ldots\e_{n}))$.
Applying Lemma 4.2 in~\cite{CG} to $D\cup E$, we have that
the first and second homology groups of $D\cup E$ with rational coefficients are trivial.
Hence, by the universal coefficient theorem for homology,
we have the first and second betti numbers of $D\cup E$ are zero.
Lemma~\ref{4mfdhom} therefore implies that
there is a prime number $p$ such that
$H_{1}(D\cup E;\mathbb{Z}_{p})\cong H_{2}(D\cup E;\mathbb{Z}_{p})\cong\{0\}$.
By the Mayer-Vietoris exact sequence, we have the following:
\[
\begin{array}{l}
\cdots\longrightarrow H_{2}(D\cap E; \mathbb{Z}_{p})\longrightarrow H_{2}(D; \mathbb{Z}_{p})\oplus H_{2}(E; \mathbb{Z}_{p})\longrightarrow \{0\}\\
\longrightarrow H_{1}(D\cap E; \mathbb{Z}_{p})\longrightarrow H_{1}(D; \mathbb{Z}_{p})\oplus H_{1}(E; \mathbb{Z}_{p})\longrightarrow \{0\}\longrightarrow\cdots.
\end{array}
\]

Since $D$ is homotopic to a point,
we have that $H_{1}(D; \mathbb{Z}_{p})\cong H_{2}(D; \mathbb{Z}_{p})\cong\{0\}$.
Therefore the homomorphism $H_{1}(D\cap E; \mathbb{Z}_{p})\rightarrow H_{1}(E; \mathbb{Z}_{p})$ is a bijection and the homomorphism $H_{2}(D\cap E; \mathbb{Z}_{p})\rightarrow H_{2}(E; \mathbb{Z}_{p})$ is a surjection.
A theorem of Stallings~\cite[5.1 Theorem]{S} implies
that the inclusion map $D\cap E\rightarrow E$ induces the following isomorphism:
\[
\frac{\pi_{1}(D\cap E)}{(\pi_{1}(D\cap E))_{q}^{p}}
\overset{\cong}{\longrightarrow}
\frac{\pi_{1}(E)}{(\pi_{1}(E))_{q}^{p}}
\]
for any natural number $q$.
Hence 
the inclusion map $\Sigma(K_{n+1})\setminus L(0\e_{2}\cdots\e_{n})\rightarrow W\setminus \widetilde{\mathcal{A}}$ induces the following isomorphism: 
\[
\frac{\pi_{1}(\Sigma(K_{n+1})\setminus L(0\e_{2}\cdots\e_{n}))}{(\pi_{1}(\Sigma(K_{n+1})\setminus L(0\e_{2}\cdots\e_{n})))_{q}^{p}}
\overset{\cong}{\longrightarrow}
\frac{\pi_{1}(W\setminus \widetilde{\mathcal{A}})}{(\pi_{1}(W\setminus \widetilde{\mathcal{A}}))_{q}^{p}}.
\]
Similarly the inclusion map $\Sigma(K'_{n+1})\setminus L'(0\e_{2}\cdots\e_{n})\rightarrow W\setminus \widetilde{\mathcal{A}}$ implies that
\[
\frac{\pi_{1}(\Sigma(K'_{n+1})\setminus L'(0\e_{2}\cdots\e_{n}))}{(\pi_{1}(\Sigma(K'_{n+1})\setminus L'(0\e_{2}\cdots\e_{n})))_{q}^{p}}
\overset{\cong}{\longrightarrow}
\frac{\pi_{1}(W\setminus \widetilde{\mathcal{A}})}{(\pi_{1}(W\setminus \widetilde{\mathcal{A}}))_{q}^{p}}.
\]
It follows that we have
\[
\frac{\pi_{1}(\Sigma(K_{n+1})\setminus L(0\e_{2}\cdots\e_{n}))}{(\pi_{1}(\Sigma(K_{n+1})\setminus L(0\e_{2}\cdots\e_{n})))_{q}^{p}}
\cong
\frac{\pi_{1}(\Sigma(K'_{n+1})\setminus L'(0\e_{2}\cdots\e_{n}))}{(\pi_{1}(\Sigma(K'_{n+1})\setminus L'(0\e_{2}\cdots\e_{n})))_{q}^{p}}.
\]
Since $L(0\e_{2}\cdots\e_{n})$ and $L'(0\e_{2}\cdots\e_{n})$ bound $\widetilde{\mathcal{A}}$,
both peripheral structures of them are preserved  by the isomorphism above.
This implies that $M_L(I)_p=M_{L'}(I)_p$. We note that 
the equation hold for any prime number which is greater than $p$. 
This and the fact that $\mu_L(I)_p$ is the modulo-$p$ reduction of $\mu_L(I)_0$ (Remark~\ref{rem}~(1)) 
complete the proof.
\end{proof}

\section{Milnor invariants and covering Milnor invariants}
\label{Brunnian}
From now on we denote $\omu_L(I)_0$ by $\omu_L(I)$ as usual. 
A link is Brunnian if any proper sublink is trivial.

\begin{theorem}
\label{thmod2}
Let $L$ be an oriented  ordered $(n+1)$-component Brunnian link in $S^{3}$ $(n\geq 2)$.
For a non-repeated sequence $I=i_1i_2\cdots i_{n+1}$   
with $i_k=n+1~(2\leq k\leq n)$, 
\[\omu_L(I)\equiv
\sum_{(\e_1,\e_2,\ldots,\e_n)\in {\mathcal{E}(I) }}\omu_{L(\e_1\e_2\cdots\e_{n})}(I\setminus\{n+1\})
\mod{2},\]
where $\mathcal{E}(I)=\{(\e_1,\e_2,\ldots,\e_n)\in {\Bbb{Z}}_2^n~|~\e_{i_{k-1}}=\e_{i_{k+1}}=0\}$ and 
$I\setminus\{n+1\}$ is a subsequence of $I$ obtained by deleting $n+1(=i_k)$.
\end{theorem}

\begin{remark}
\label{bestpossible}
There is a $3$-component Brunnian link (see Figure~\ref{ex1}) such that $\omu_{L}(132)=-1$ and $\omu_{L(00)}(12)=1$.
Hence
Theorem~\ref{thmod2} 
does not hold without taking modulo $2$.
\end{remark}

\begin{remark}
Since the image of $L(\e_1\e_2\cdots\e_n)\setminus K^{\e_i}_i$ is a trivial link that bounds a 
disjoint union of disks in $S^3\setminus K_{n+1}$, 
$L(\e_1\e_2\cdots\e_n)$ is also a Brunnian link in $\Sigma(K_{n+1})(=S^{3})$.
Hence we can repeatedly apply Theorem~\ref{thmod2} to covering links.
Then we have that 
the length $n+1$ Milnor invariants for $L$ is modulo-$2$ congruent to a sum of linking numbers of 
\lq iterated' covering links.
\end{remark}

\section{Claspers}
We use clasper theory introduced by K. Habiro~\cite{H} to prove Theorem~\ref{thmod2}.
In this section, 
we briefly recall from \cite{H} the basic notions of clasper theory.
In this paper, we essentially only need the notion of $C_k$-tree.
For a general definition of claspers, we refer the reader to \cite{H}.
Let $L$ be a link in $S^{3}$.
\begin{definition}
An embedded disk $T$ in $S^{3}$ is called a {\it tree clasper} for $L$ if it satisfies the following (1), (2) and (3):\\
(1) $T$ is decomposed into disks and bands, called {\it edges}, each of which connects two distinct disks.\\
(2) The disks have either 1 or 3 incident edges, called {\it disk-leaves} or {\it nodes} respectively.\\
(3) $L$ intersects $T$ transversely and the intersections are contained in the union of the interior of the disk-leaves.\\
The {\it degree} of a tree clasper is the number of the disk-leaves minus 1.
(In \cite{H}, a tree clasper is called a {\it strict tree clasper}.)
A degree $k$ tree clasper is called a {\it $C_k$-tree}.
A $C_k$-tree is {\it simple} if each disk-leaf intersects $L$ at one point.
\end{definition}
We will make use of the drawing convention for claspers of Fig. 4 in~\cite{H}.
Given a $C_k$-tree $T$ for $L$, there is a procedure to construct a framed link $\gamma(T)$ in a regular neighborhood of $T$.
{\it Surgery along $T$} means surgery along $\gamma(T)$.
Since surgery along $\gamma(T)$ preserves the ambient space, surgery along the $C_k$-tree $T$ can be regarded as a local move on $L$ in $S^{3}$.
We will denote by $L_{T}$ the link in $S^{3}$ which is obtained from $L$ 
by surgery along $T$.
Similarly, for a disjoint union of tree claspers $T_1\cup T_{2}\cup\cdots \cup T_m$, we can define $L_{T_1\cup T_{2}\cup\cdots \cup T_m}$.
A $C_k$-tree $T$ having the shape of tree clasper like Figure~\ref{linear} is called a \textit{linear $C_k$}-tree, and the leftmost and rightmost disk-leaves of $T$ are called the \textit{ends} of $T$.
In particular, surgery along a simple linear $C_{k}$-tree for $L$ is ambient isotopic to a band summing of $L$ and the $(k+1)$-component Milnor link~(Fig.~7 in \cite{M}), see Figure~\ref{linear}.

\begin{figure}[htbp]
\begin{center}
 \begin{overpic}[width=120mm]{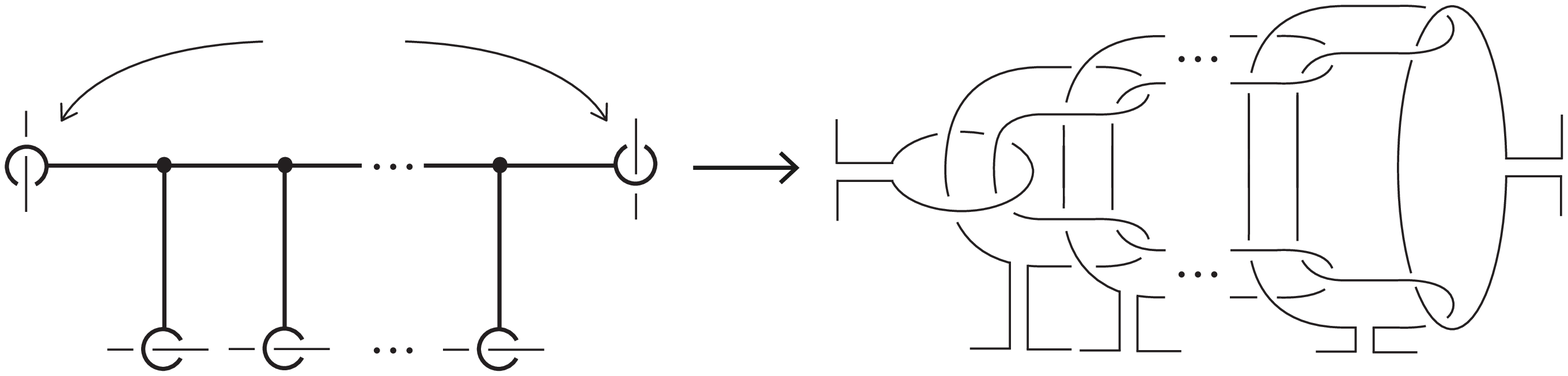}
   \put(64,73){ends}
   \put(146,55){surgery}
 \end{overpic}
 \caption{Surgery along a simple linear $C_{k}$-tree}
 \label{linear}
\end{center}
\end{figure}

\begin{definition}
A simple $C_k$-tree $T$ 
for $L=K_{1}\cup K_{2}\cup\cdots\cup K_{n}$ is a \textit{$C_k^d$-tree} if $|\{\ i\ | K_{i} \cap \ T \neq \emptyset, i=1,2,\ldots,n\}| =k+1$.
Let $T$ be a linear $C_k^{d}$-tree with the ends $f_0$ and $f_k$.
Since $T$ is a disk, we can travel from $f_0$ to $f_k$ along the boundary of $T$ so that we meet all other disk-leaves $f_1,\ldots,f_{k-1}$ in this order.
If $f_s$ intersects the $i_s$th component of $L$ $(s=0,\ldots,k)$, we can consider two vectors $(i_0,\ldots,i_k)$ and $(i_k,\ldots,i_0)$, and may assume that $(i_0,\ldots,i_k)$ $\leq$ $(i_k,\ldots,i_0)$, where \lq$\leq$' is the lexicographic order in $\mathbb{Z}^{k+1}$.
We call $(i_0,\ldots,i_k)$ the \textit{o-index} of $T$.
\end{definition}

The following theorem is essentially shown by Milnor~\cite{M}.
\begin{theorem}\cite[Section 5]{M}
\label{thMilnor}
Let $O=O_{1}\cup O_{2}\cup\cdots\cup O_{n+1}$ be an oriented $(n+1)$-component trivial link in $S^{3}$, and 
$T$ a linear $C_{n}^{d}$-tree for $O$ with the o-index $(i_{1},i_{2},\ldots,i_{n+1})$.
Then all Milnor invariants of $O_{T}$ with the length $\leq n$ vanish and 
for a non-repeated sequence $i_{1}j_{2}\cdots j_{n}i_{n+1}$ of $\{1,2,\ldots,n+1\}$
\[
\left|\omu_{O_{T}}(i_{1}j_{2}\cdots j_{n}i_{n+1})\right|=\Big\{
\begin{array}{rcl}
1&\textrm{if}~ j_2j_3\cdots j_{n}=i_2i_3\cdots i_{n},\\
0&\textrm{otherwise}.
\end{array}
\]
\end{theorem}

The following lemma is shown similarly to  Lemma~2.9 in~\cite{Meilhan}.
\begin{lemma}
\label{lemIHX}
Let $L$ be a link,
and let $T_{I}$, $T_{H}$ and $T_{X}$ be $C_{k}$-trees for $L$ 
which differ only in a small ball as illustrated in Figure~{\rm \ref{IHX}}.
Then there are $C_k$-trees $T'_H$ and $T'_X$ such that 
$L_{T'_{H}\cup T'_{X}}$ is ambient isotopic to  $L_{T_{I}}$, and  
$L\cup T'_{H}\cup T'_{X}$ is obtained from $L\cup T_{H}\cup T_{X}$ by 
changing crossings among edges of the claspers and $L$.
\end{lemma}

\begin{figure}[htbp]
\begin{center}
 \begin{overpic}[width=60mm]{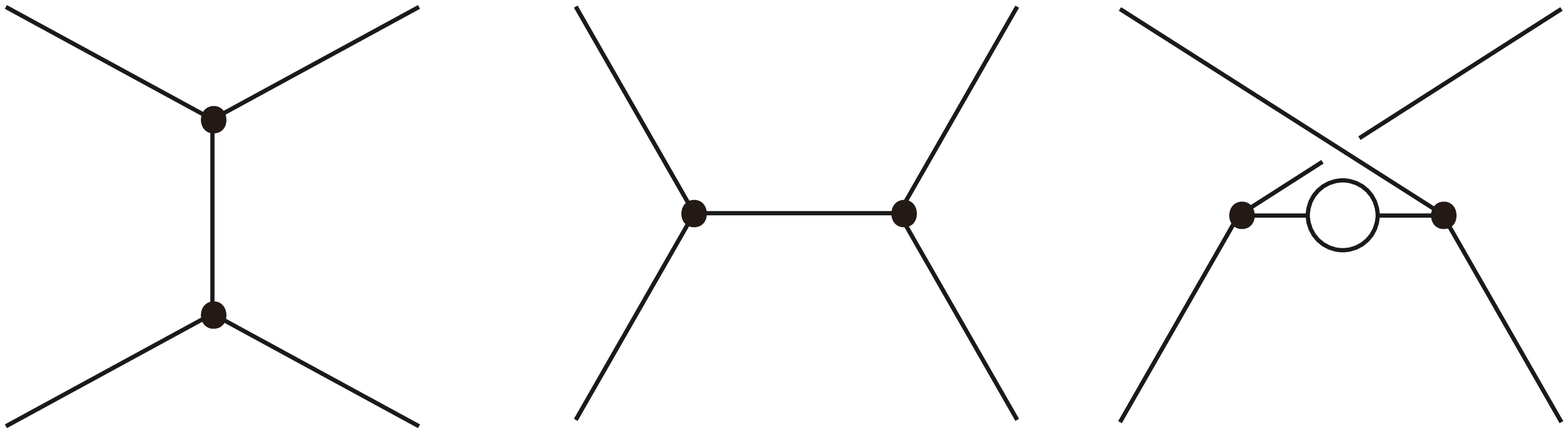}
 \put(18,-5){$T_{I}$}
 \put(82,-5){$T_{H}$}
 \put(141,-5){$T_{X}$}
 \put(144.1,21){$s$}
 \end{overpic}
 \caption{The IHX relation for $C_{k}$-trees}
 \label{IHX}
\end{center}
\end{figure}

\section{Proof of Theorem~\ref{thmod2}}
\label{proof}
Let $L=K_1\cup K_{2}\cup\cdots\cup K_{n+1}$ be an oriented $(n+1)$-component Brunnian link in $S^{3}$. 
Since for a non-repeated sequence $I=i_1i_2\cdots i_{n+1}$ in Theorem~\ref{thmod2}, the essential condition is that $n+1$ is neither 
$i_1$ nor $i_{n+1}$, we may assume that $i_1=1$ and $i_{n+1}=2$ to avoid complicated arguments in this section. 
 
Let $\Sigma(K_{n+1})$ be the double branched cover of $S^3$ branched over $K_{n+1}$. 
It is shown in~\cite{MY}, \cite{H2} that there is a disjoint union of 
linear $C_n^d$-trees $T_1\cup T_{2}\cup\cdots\cup T_m$ for an $(n+1)$-component trivial link $O=O_1\cup O_{2}\cup\cdots\cup O_{n+1}$ such that $L$ and 
$O_{T_1\cup T_{2}\cup\cdots\cup T_m}$ are ambient isotopic.
Consider induction on the length of the path connecting two disk-leaves of each $T_{r}$ grasping $O_{1}$ and $O_{2}$, 
Lemma~\ref{lemIHX} implies the following.

\begin{proposition}
Let $L$ be an $(n+1)$-component Brunnian link in $S^{3}$.
There is a disjoint union of 
linear $C_n^d$-trees $T_1\cup T_{2}\cup\cdots\cup T_m$ for an $(n+1)$-component trivial link $O=O_1\cup O_{2}\cup\cdots\cup O_{n+1}$ such that $L$ and 
$O_{T_1\cup T_{2}\cup\cdots\cup T_m}$ are ambient isotopic and 
that the ends of $T_{r}$ grasp $O_{1}$ and $O_{2}$ for each $r(=1,2,\ldots,m)$.
\end{proposition}

\begin{figure}[htbp]
\begin{center}
\begin{tabular}{c}
\begin{minipage}{0.5\hsize}
\begin{center}
 \begin{overpic}[width=45mm]{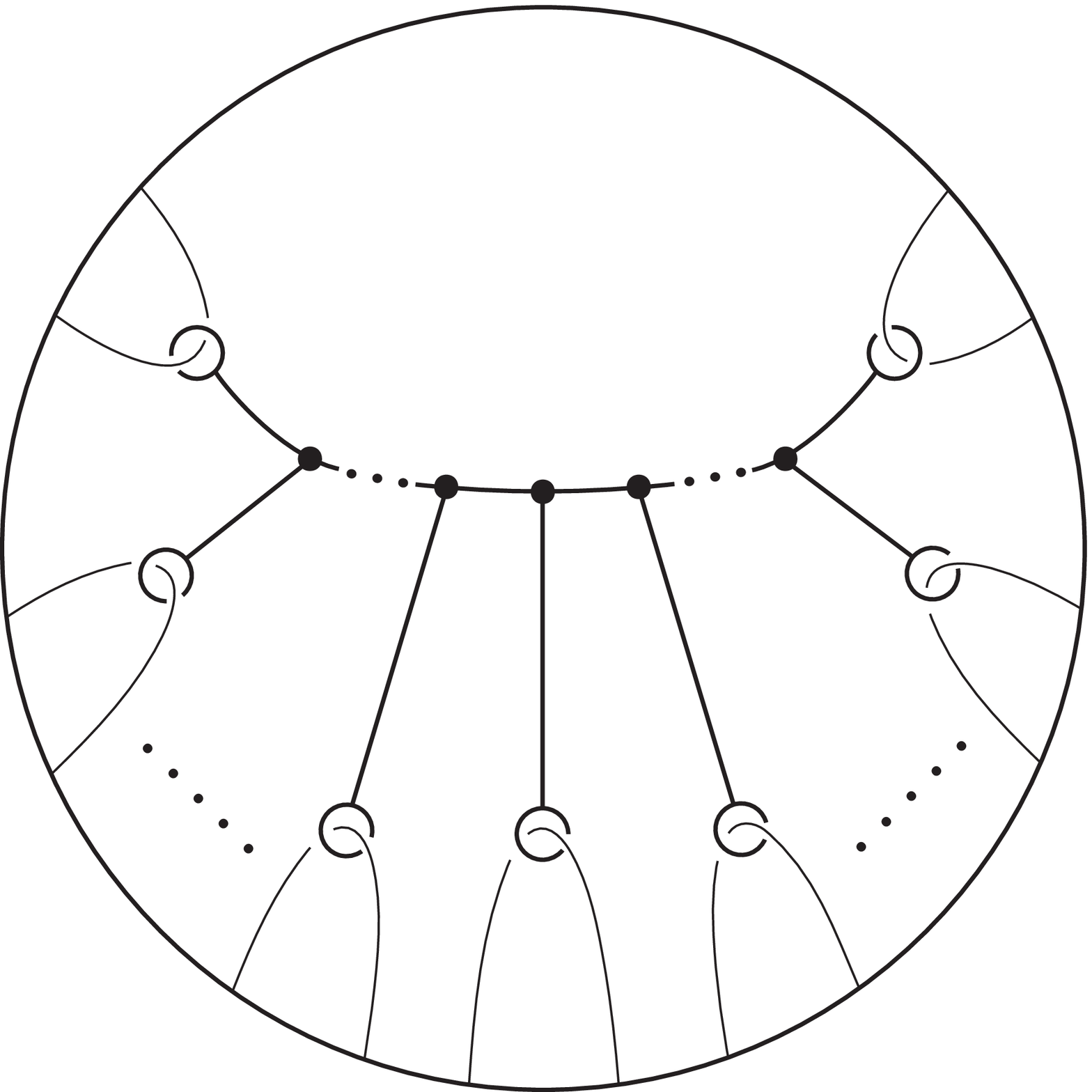}
   \put(60,78){$T_{r}$}
   \put(60,131){$B_{r}$}
   \linethickness{0.5pt}
   \put(142,59){\vector(1,0){30}}
   \put(142,63){move $9$}
   \put(-3,100){$O_{1}$}
   \put(118,100){$O_{2}$}
   \put(59,-20){(a)}
   \put(62,-9){$\alpha (\subset O_{n+1})$}
 \end{overpic}
\end{center}
\end{minipage}
\begin{minipage}{0.5\hsize}
\begin{center}
 \begin{overpic}[width=45mm]{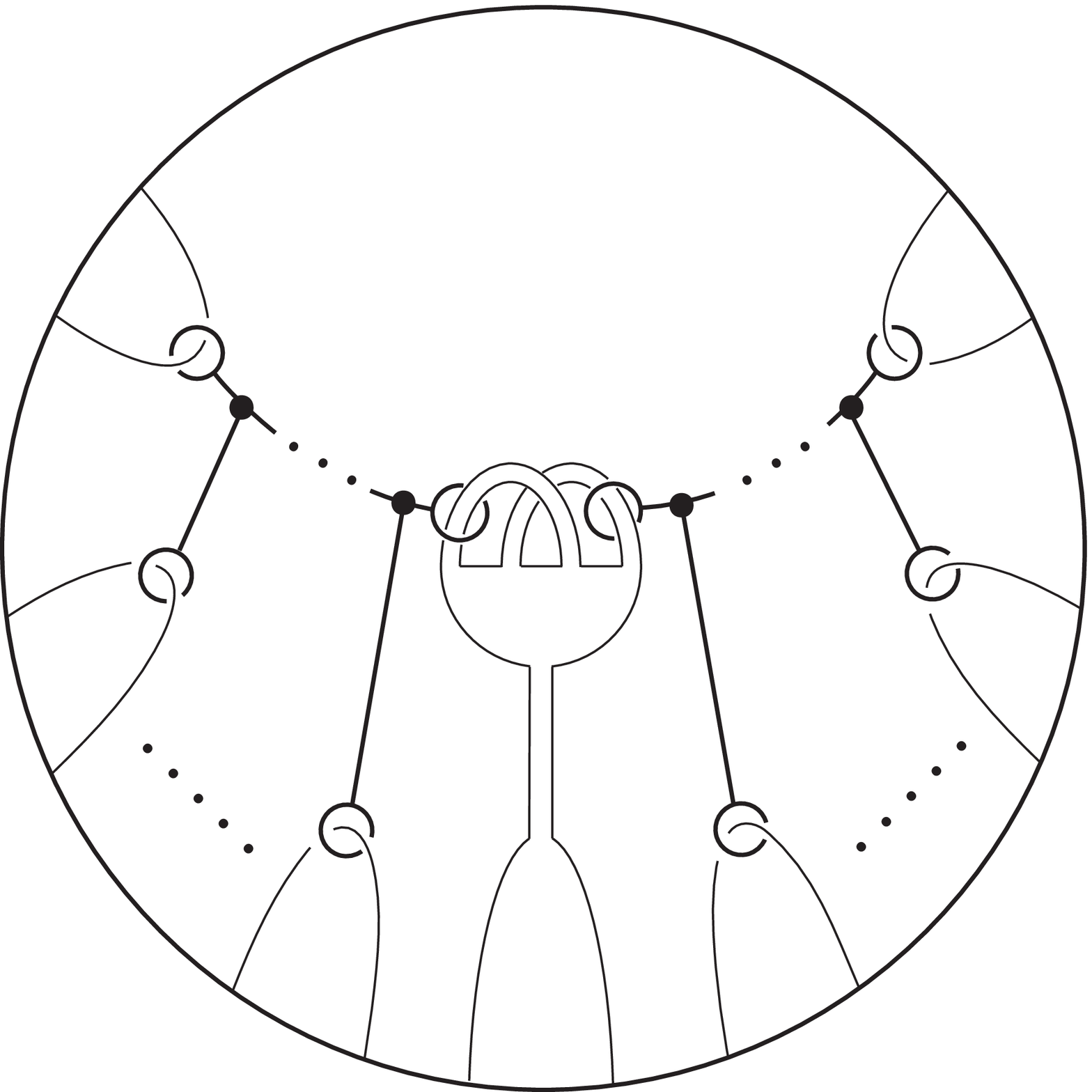}
 \put(60,10){$F$}
   \put(60,131){$B_{r}$}
   \put(35,81){$T_{r1}$}
   \put(81,81){$T_{r2}$}
   \put(-3,100){$O_{1}$}
   \put(118,100){$O_{2}$}
   \put(57.5,-20){(b)}
 \end{overpic}
\end{center}
\end{minipage}
\end{tabular}
\end{center}
\caption{}
 \label{F4}
\end{figure}

We identify $K_{n+1}$ and $O_{n+1}$. 
In the following, we will observe that the preimage of $K_1\cup K_{2}\cup\cdots\cup K_n$ is obtained from the preimage of 
$O_1\cup O_{2}\cup\cdots\cup O_n$ by surgery along claspers. 
There is a disjoint union of 3-balls $B_1\cup B_{2}\cup\cdots\cup B_m$ in $S^3$ such that 
$B_r\cap(T_1\cup T_{2}\cup\cdots\cup T_m)=T_r$ and 
$B_{r}\cap O$ is a trivial tangle.
Let $\alpha=B_r\cap O_{n+1}$, see Figure~\ref{F4} (a).
We consider the double branched cover of $B_r$ branched over $\alpha$.
By using move 9 in \cite{H}, we decompose $T_r$ into two tree claspers $T_{r1}$ and $T_{r2}$ 
(see Figure~\ref{F4} (b)), where $T_{rj}$ intersects $O_j$ $(j=1,2)$. 
Recall that the ends of $T_r$ grasp $O_1$ and $O_2$. 
Set $O_{j_{s}}^{\de_{s}}$ 
(resp. $O_{j}^{\e_{j}}$) 
be a preimage of $O_{j_{s}}$ 
(resp. $O_{j}$)
for $s=1,2,\ldots,n-2, \de_{s}(=\e_{j_s})\in\mathbb{Z}_{2}$ 
(resp. for $j=1,2, \e_{j}\in\mathbb{Z}_{2}$).
By using genus-1 surface $F$ as illustrated in Figure~\ref{F4}~(b), we have a surgery description of 
the double branched cover as illustrated in Figure~\ref{F5}~\cite{AK}. 
Let $T_{r1}^{\e}$ (resp. $T_{r2}^{\e}$) be a preimage of $T_{r1}$ with 
$T_{r1}^{\e}\cap O_{j_{l}}^{\e}\neq\emptyset$ (resp. $T_{r2}^{\e}\cap O_{j_{l+1}}^{\e}\neq\emptyset$) for $\e\in\{0,1\}$.
Then we have a new clasper $G_r$ with two boxes, see Figure~\ref{F6}.
Here a {\it box}, given in~\cite{H}, is a disk with $3$ incident edges 
which is obtained by replacing a disjoint union of $3$ disk-leaves as illustrated in Figure~\ref{box}.

\begin{figure}[htbp]
\begin{center}
 \begin{overpic}[width=90mm]{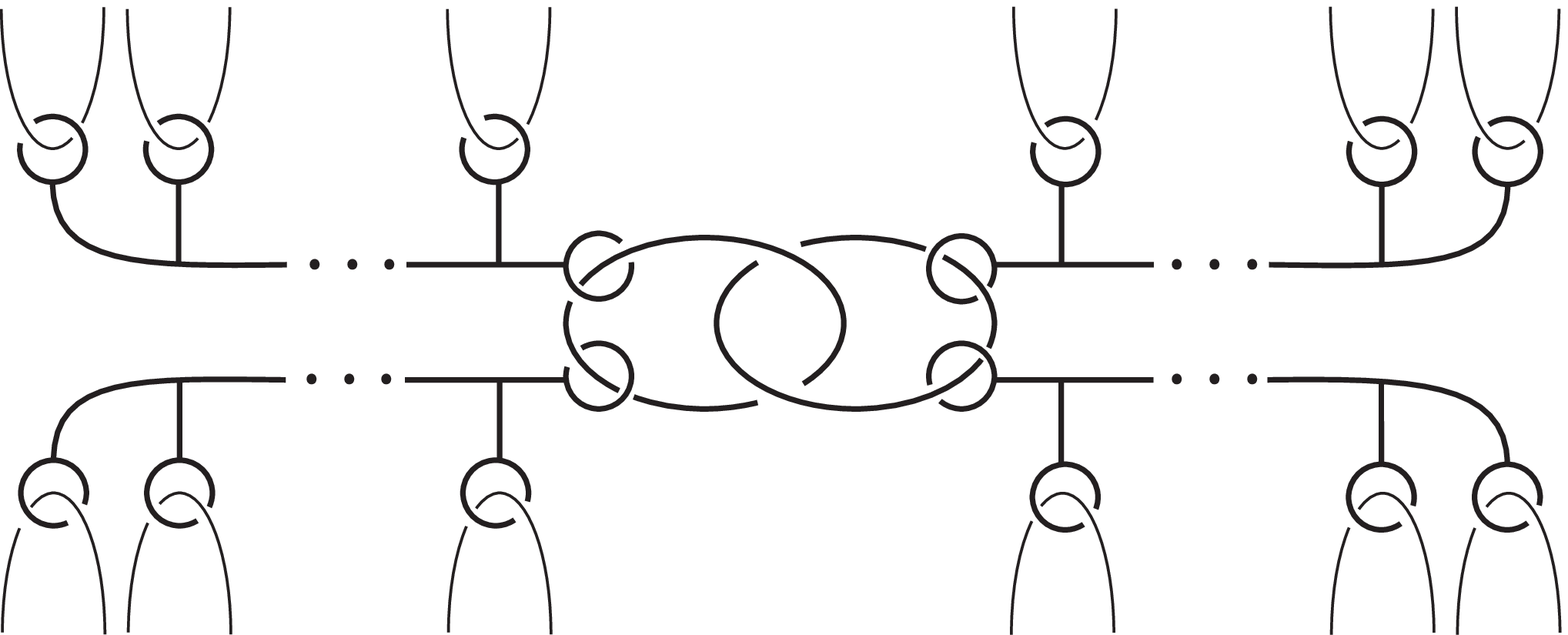}
   \put(-16,95){$O_{1}^{\e_{1}}$}
   \put(39,95){$O_{j_{1}}^{\de_{1}}$}
   \put(92,95){$O_{j_{l}}^{\de_{l}}$}
   \put(140,95){$O_{j_{l+1}}^{\de_{l+1}}$}
   \put(188,95){$O_{j_{n-2}}^{\de_{n-2}}$}
   \put(257,95){$O_{2}^{\e_{2}}$}
   \put(-25,3){$O_{1}^{\e_{1}+1}$}
   \put(39,3){$O_{j_{1}}^{\de_{1}+1}$}
   \put(91,3){$O_{j_{l}}^{\de_{l}+1}$}
   \put(133,3){$O_{j_{l+1}}^{\de_{l+1}+1}$}
   \put(183,3){$O_{j_{n-2}}^{\de_{n-2}+1}$}
   \put(257,3){$O_{2}^{\e_{2}+1}$}
   \put(113,67){$0$}
   \put(138,67){$0$}
   \put(-9,60){$T_{r1}^{\de_{l}}$}
   \put(248,60){$T_{r2}^{\de_{l+1}}$}
   \put(-9,37){$T_{r1}^{\de_{l}+1}$}
   \put(248,37){$T_{r2}^{\de_{l+1}+1}$}
 \end{overpic}
 \caption{A surgery description of the double branched cover}
 \label{F5}
\end{center}
\end{figure}
\begin{figure}[htbp]
\begin{center}
\begin{overpic}[width=90mm]{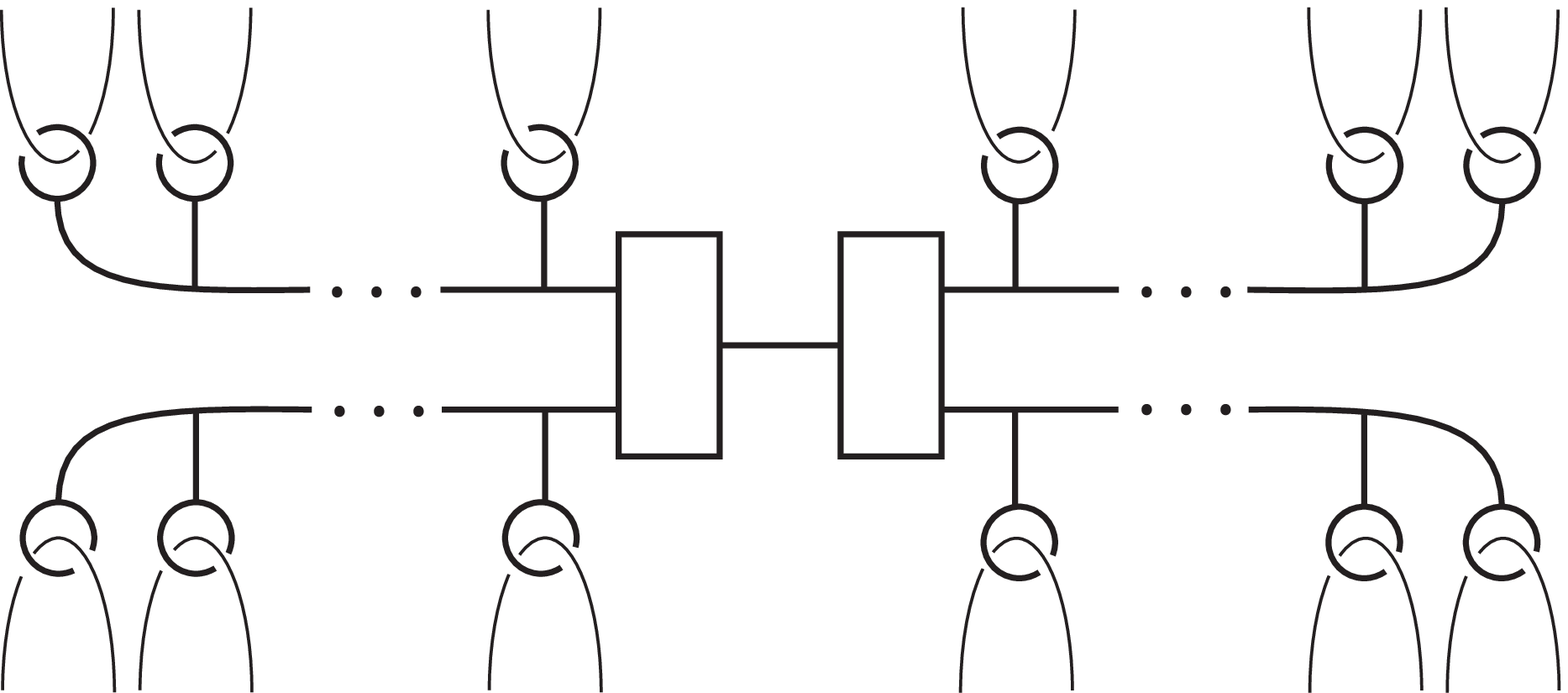}
   \put(-16,103){$O_{1}^{\e_{1}}$}
   \put(43,103){$O_{j_{1}}^{\de_{1}}$}
   \put(100,103){$O_{j_{l}}^{\de_{l}}$}
   \put(131,103){$O_{j_{l+1}}^{\de_{l+1}}$}
   \put(185,103){$O_{j_{n-2}}^{\de_{n-2}}$}
   \put(257,103){$O_{2}^{\e_{2}}$}
   \put(-25,3){$O_{1}^{\e_{1}+1}$}
   \put(44,3){$O_{j_{1}}^{\de_{1}+1}$}
   \put(99,3){$O_{j_{l}}^{\de_{l}+1}$}
   \put(124,3){$O_{j_{l+1}}^{\de_{l+1}+1}$}
   \put(179,3){$O_{j_{n-2}}^{\de_{n-2}+1}$}
   \put(257,3){$O_{2}^{\e_{2}+1}$}
   \put(-10,66){$T_{r1}^{\de_{l}}$}
   \put(248,66){$T_{r2}^{\de_{l+1}}$}
   \put(-10,41){$T_{r1}^{\de_{l}+1}$}
   \put(248,41){$T_{r2}^{\de_{l+1}+1}$}
   \put(122,78){$G_{r}$}
\end{overpic}
 \caption{}
 \label{F6}
\end{center}
\end{figure}

\begin{figure}[htbp]
\begin{center}
 \begin{overpic}[width=60mm]{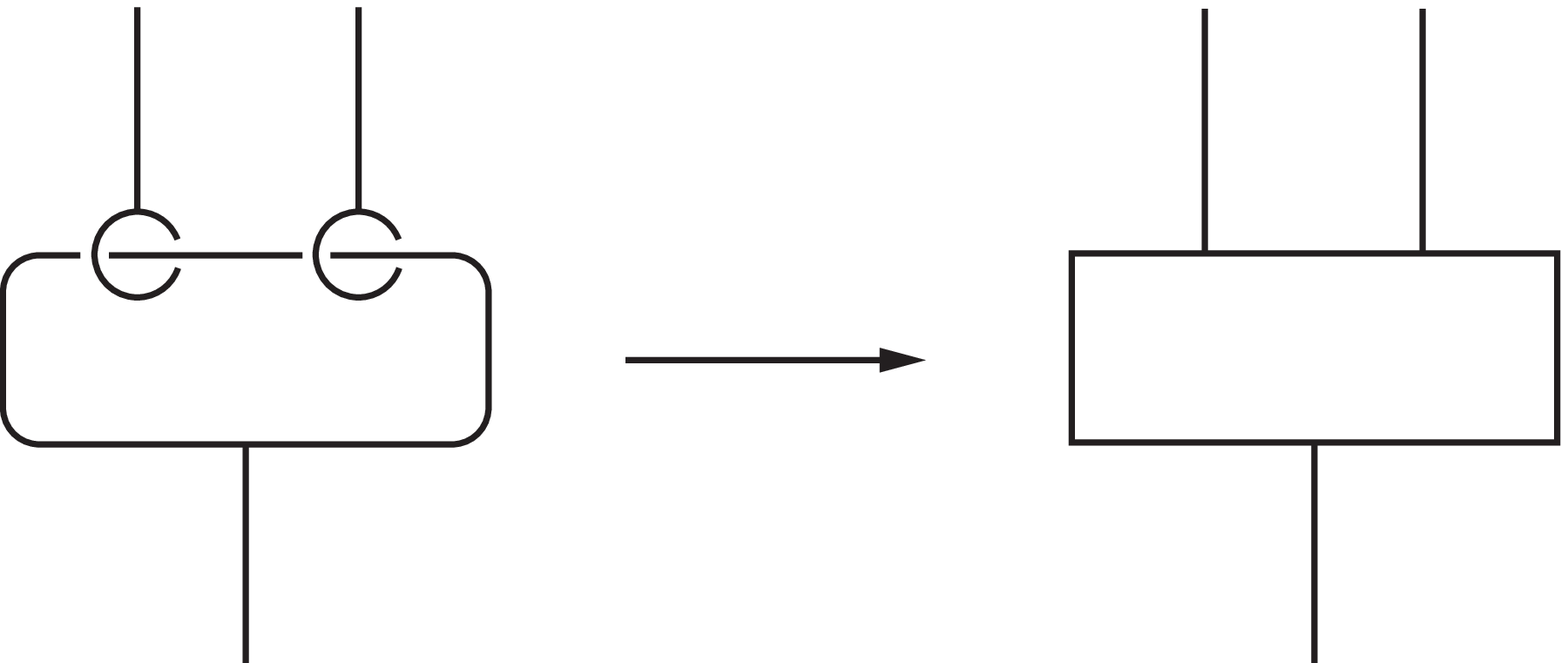}
 \end{overpic}
 \caption{A box}
 \label{box}
\end{center}
\end{figure}

The following lemma is shown by Habiro \cite[Proposition 3.4]{H}.
 
\begin{lemma}\cite[Proposition 3.4]{H}
\label{prop3.4}
Let $T$ be a tree clasper for a link $L$ such that 
$T$ has a disk-leaf without intersecting $L$. Then $L_T$ is ambient isotopic to $L$.
 \end{lemma}

\begin{lemma}
\label{lemmod2}
Let $T_{r}$ be a linear $C_{n}^{d}$-tree for $O=O_{1}\cup O_{2}\cup\cdots\cup O_{n+1}$ with the o-index $(1,j_{2},\ldots,j_{n},2)$.
Then for a sequence $I=1i_{2}i_{3}\cdots i_{n}2$ with $i_{k}=n+1$, the following holds:
\[
\sum_{(\e_{1},\e_{2},\ldots,\e_{n})\in\mathcal{E}(I)}
\left|\omu_{(O_{1}^{\e_{1}}\cup O_{2}^{\e_{2}}\cup\cdots\cup O_{n}^{\e_{n}})_{G_{r}}}(I\setminus\{n+1\})\right|=\Big\{
\begin{array}{rcl}
&1&\textrm{if}~ {1}j_{2}\cdots j_{n}2=I,\\
&0 \textrm{ or }  2&\textrm{otherwise}.
\end{array}
\]
\end{lemma}

\begin{proof}
Suppose that ${1}j_{2}\cdots j_{n}2=I$.
If $(\e_{1},\e_{2},\ldots,\e_{n})\in\mathcal{E}(I)$,
then $\e_{i_{k-1}}=\e_{i_{k+1}}=0$, and hence  by 
Lemma~\ref{prop3.4} $(O_{1}^{\e_{1}}\cup\ O_{2}^{\e_{2}}\cup\cdots\cup O_{n}^{\e_{n}})_{G_{r}}$ 
is ambient isotopic to 
$(O_{1}^{\e_{1}}\cup\ O_{2}^{\e_{2}}\cup\cdots\cup O_{n}^{\e_{n}})_{(G_{r}\setminus (T_{r1}^{1}\cup T_{r2}^{1}))}$. 
Applying move 2 in \cite{H} to 
$G_{r}\setminus (T_{r1}^{1}\cup T_{r2}^{1})$, 
we have a linear $C_{n-1}^{d}$-tree $\widetilde{T_{r}}$ for $O_{1}^{\e_{1}}\cup O_{2}^{\e_{2}}\cup\cdots\cup O_{n}^{\e_{n}}$.
If there is an index $i\in\{1,2,\ldots,n\}$ such that $O_{i}^{\e_{i}}\cap \widetilde{T_{r}}=\emptyset$,
then $(O_{1}^{\e_{1}}\cup O_{2}^{\e_{2}}\cup\cdots\cup O_{n}^{\e_{n}})_{\widetilde{T_{r}}}$
is a trivial link by Lemma~\ref{prop3.4}.
Since there is a unique choice $(\e_1,\e_2,\ldots,\e_n)\in{\mathcal{E}}(I)$ such  that $O_{i}^{\e_{i}}\cap\widetilde{T_{r}}\neq\emptyset$ for any $i$, 
and since the o-index  of $\widetilde{T_{r}}$ is equal to $(1,i_{2},\ldots, i_{k-1},i_{k+1},\ldots, i_{n},2)$, 
then by Theorem~\ref{thMilnor} we have that
\[
\sum_{(\e_{1},\e_{2},\ldots,\e_{n})\in\mathcal{E}(I)}
\left|\omu_{(O_{1}^{\e_{1}}\cup O_{2}^{\e_{2}}\cup\cdots\cup O_{n}^{\e_{n}})_{G_r}}(I\setminus\{n+1\})\right|=1.\]

Suppose that ${1}j_{2}\cdots j_{n}2\neq I$ and ${1}j_{2}\cdots j_{n}2\setminus\{n+1\}=I\setminus\{n+1\}$.
If $(\e_{1},\e_{2},\ldots,\e_{n})\in\mathcal{E}(I)$, then 
$G_r\cap (O_{i_{k-1}}^0\cup O_{i_{k+1}}^0 )$ is contained in 
$T_{r1}^{0}\cup T_{r1}^{1}$ or $T_{r2}^{0}\cup T_{r2}^{1}$. 
Without loss of generality, we may assume that 
$T_{r1}^{0}\cup T_{r1}^{1}$ contains $G_r\cap (O_{i_{k-1}}^0\cup O_{i_{k+1}}^0 )$.
If both $T_{r1}^{0}$ and $T_{r1}^{1}$ intersect  $O_{i_{k-1}}^0\cup O_{i_{k+1}}^0$, then by Lemma~\ref{prop3.4}, 
$(O_{1}^{\e_{1}}\cup O_{2}^{\e_{2}}\cup\cdots\cup O_{n}^{\e_{n}})_{G_{r}}$ is  trivial.  It follows that 
\[
\sum_{(\e_{1},\e_{2},\ldots,\e_{n})\in\mathcal{E}(I)}
\left|\omu_{(O_{1}^{\e_{1}}\cup O_{2}^{\e_{2}}\cup\cdots\cup O_{n}^{\e_{n}})_{G_r}}(I\setminus\{n+1\})\right|=0.\]

\noindent
Suppose that either 
$T_{r1}^0$ or $T_{r1}^1$ contains $G_r\cap (O_{i_{k-1}}^0\cup O_{i_{k+1}}^0 )$. 
Here we may assume that $G_r\cap (O_{i_{k-1}}^0\cup O_{i_{k+1}}^0 )\subset T_{r1}^0$.  
Then $(O_{1}^{\e_{1}}\cup O_{2}^{\e_{2}}\cup\cdots\cup O_{n}^{\e_{n}})_{G_{r}}$ is ambient isotopic to 
 $(O_{1}^{\e_{1}}\cup O_{2}^{\e_{2}}\cup\cdots\cup O_{n}^{\e_{n}})_{G_{r}\setminus T_{r1}^{1}}$. 
 We note that by Lemma~\ref{prop3.4} there are exactly two possibilities  $(\e_1,\e_2,\ldots,\e_n)\in{\mathcal{E}}(I)$ such that 
  $(O_{1}^{\e_{1}}\cup O_{2}^{\e_{2}}\cup\cdots\cup O_{n}^{\e_{n}})_{G_{r}\setminus T_{r1}^{1}}$ 
 is a nontrivial link.  
 These two choice give us 
 $(O_{1}^{\e_{1}}\cup O_{2}^{\e_{2}}\cup\cdots\cup O_{n}^{\e_{n}})_{G_{r}\setminus T_{r1}^{1}}$ 
which  is ambient isotopic to either 
 $(O_{1}^{\e_{1}}\cup O_{2}^{\e_{2}}\cup\cdots\cup O_{n}^{\e_{n}})_{G_{r}\setminus (T_{r1}^{1}\cup T_{r2}^0)}$ or  
 $(O_{1}^{\e_{1}}\cup O_{2}^{\e_{2}}\cup\cdots\cup O_{n}^{\e_{n}})_{G_{r}\setminus (T_{r1}^{1}\cup T_{r2}^1)}$. 
 Applying move 2 in~\cite{H} to $G_{r}\setminus(T_{r1}^{1}\cup T_{r2}^{\e})$ for $\e\in\{0,1\}$, 
we obtain two linear $C_{n-1}^{d}$-trees with the o-indexes $(1,i_{2},\ldots, i_{k-1},i_{k+1},\ldots, i_{n},2)$. 
Hence Theorem~\ref{thMilnor} implies that 
\[
\sum_{(\e_{1},\e_{2},\ldots,\e_{n})\in\mathcal{E}(I)}
\left|\omu_{(O_{1}^{\e_{1}}\cup O_{2}^{\e_{2}}\cup\cdots\cup O_{n}^{\e_{n}})_{G_r}}(I\setminus\{n+1\})\right|=2.\]

Suppose that $1j_{2}\cdots j_{n}2\setminus\{n+1\}\neq I\setminus\{n+1\}$. 
Let $(\e_1,\e_2,\ldots,\e_n)\in {\mathcal{E}}(I)$ be a choice such that 
$(O_{1}^{\e_{1}}\cup O_{2}^{\e_{2}}\cup\cdots\cup O_{n}^{\e_{n}})_{G_{r}}$ is nontrivial, that is, by Lemma~\ref{prop3.4}
it is ambient isotopic to  
$(O_{1}^{\e_{1}}\cup O_{2}^{\e_{2}}\cup\cdots\cup O_{n}^{\e_{n}})_{G_{r}\setminus (T_{r1}^{\e}\cup T_{r2}^{\e'})}$ 
for some $\e,\e'\in {\mathbb{Z}}_2$. 
 Applying move 2 in~\cite{H} to $G_{r}\setminus(T_{r1}^{\e}\cup T_{r2}^{\e'})$, 
we obtain a linear $C_{n-1}^{d}$-tree whose the o-index is not equal to $(1,i_{2},\ldots, i_{k-1},i_{k+1},\ldots, i_{n},2)$. 
Hence Theorem~\ref{thMilnor} implies that 
\[
\sum_{(\e_{1},\e_{2},\ldots,\e_{n})\in\mathcal{E}(I)}
\left|\omu_{(O_{1}^{\e_{1}}\cup O_{2}^{\e_{2}}\cup\cdots\cup O_{n}^{\e_{n}})_{G_r}}(I\setminus\{n+1\})\right|=0.\]
\end{proof}

\begin{proof}[Proof of Theorem~{\rm \ref{thmod2}}]
It is shown in~\cite{Cochran} that 
the length-$l$ Milnor invariants are additive under the band sum operation
if all Milnor invariants with the length $\leq l-1$ vanish.
The boundary $\partial B_r$ of the regular neighborhood $B_r$ of $T_r$ 
can be assumed to be a decomposing sphere for a band sum operation between 
 $O_{T_r}$ and $O_{(T_1\cup T_{2}\cup\cdots\cup T_m)
 \setminus T_r}$. 
Hence we may assume that $L$ is a band sum of $m+1$ links $O_{T_1},O_{T_{2}},\ldots,O_{T_m}$ and $O$.
Since $O_{T_r}$ is an $(n+1)$-component Brunnian link for each $r(=1,2,\ldots,m)$,
all Milnor invariants with the length $\leq n$ vanish.
Therefore we have that
\[\omu_L(I)=\sum_{r=1}^m\omu_{O_{T_r}}(I).\]
By combining Theorem~\ref{thMilnor} and Lemma~\ref{lemmod2}, we have 
\[\sum_{r=1}^m\omu_{O_{T_r}}(I)\equiv 
\sum_{r=1}^m\sum_{(\e_{1},\e_{2},\ldots,\e_{n})\in\mathcal{E}(I)}
\omu_{(O_{1}^{\e_{1}}\cup O_{2}^{\e_{2}}\cup\cdots\cup O_{n}^{\e_{n}})_{G_r}}(I\setminus\{n+1\}) ~\mathrm{~~mod~}2.\]
Since $L(\e_{1}\e_{2}\cdots\e_{n})$ is also Brunnian, we have that 
\[\omu_{L(\e_{1}\e_{2}\cdots\e_{n})}(I\setminus\{n+1\})=
\sum_{r=1}^m\omu_{(O_{1}^{\e_{1}}\cup O_{2}^{\e_{2}}\cup\cdots\cup O_{n}^{\e_{n}})_{G_r}}(I\setminus\{n+1\}).\]
This completes the proof.
\end{proof}

\section{Link-homotopy}
Link-homotopy is an equivalence relation generated by crossing changes on the same component of links.
It is shown in~\cite{M} that 
Milnor invariants for non-repeated sequences are link-homotopy invariants of links.
However for covering Milnor invariants the same result does not hold.  

\begin{figure}[htbp]
\begin{center}
\begin{tabular}{c}
\begin{minipage}{0.4\hsize}
\begin{center}
 \begin{overpic}[width=40mm]{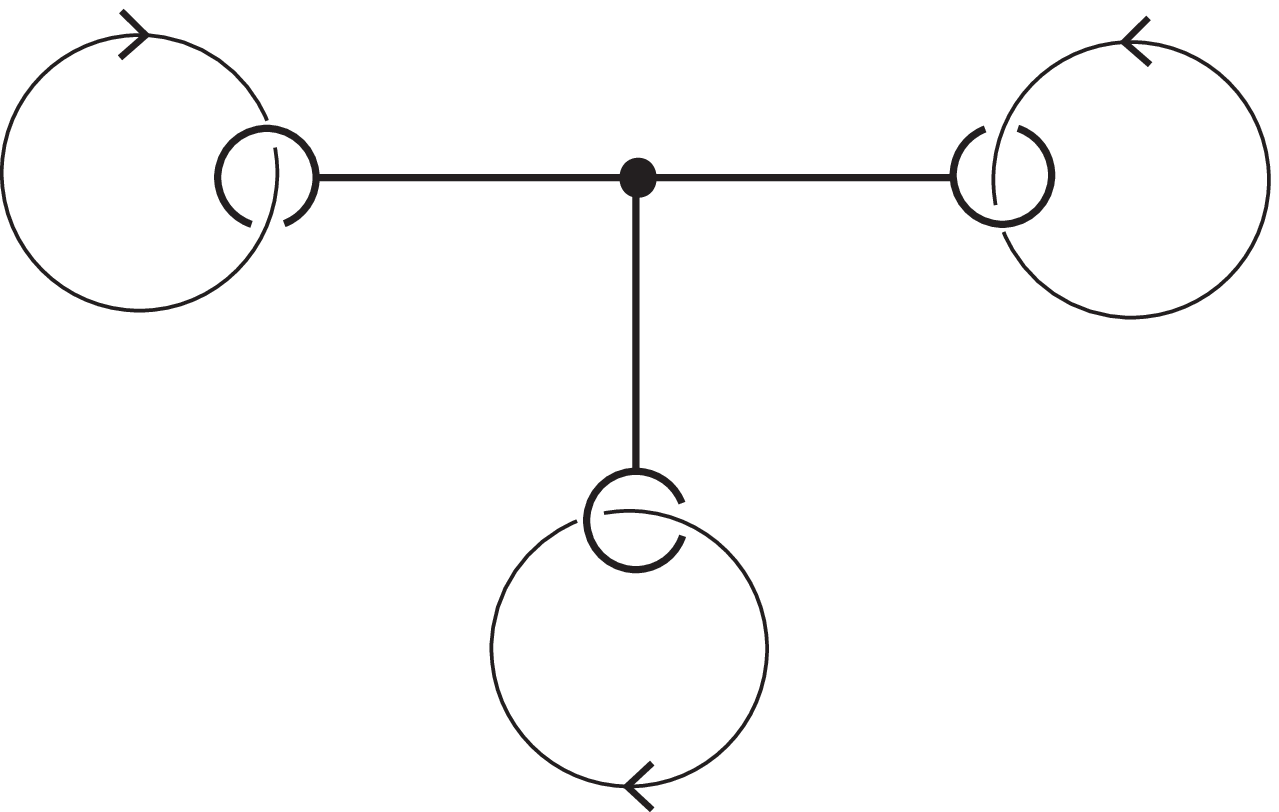}
  \linethickness{0.5pt}
   \put(146,29){\vector(-1,0){20}}
   \put(52,65){$T_{1}$}
   \put(-14,50){$O_{1}$}
   \put(117,50){$O_{2}$}
   \put(72,8){$O_{3}$}
   \put(52,-10){(a)}
 \end{overpic}
\end{center}
\end{minipage}

\begin{minipage}{0.6\hsize}
\vspace{7.5mm}\begin{center}
 \begin{overpic}[width=56mm]{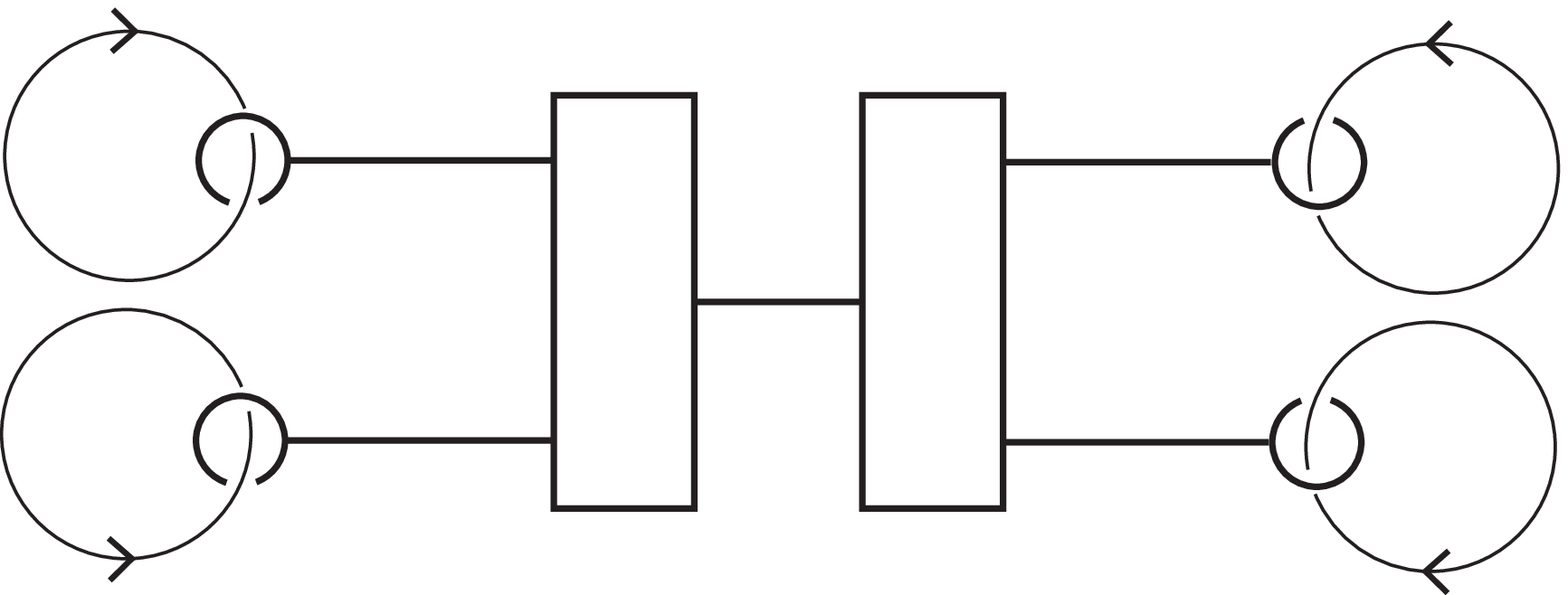}
  \linethickness{0.5pt}
   \put(-14,49){$O_{1}^{0}$}
   \put(-14,4){$O_{1}^{1}$}
   \put(161,49){$O_{2}^{0}$}
   \put(161,4){$O_{2}^{1}$}
   \put(74,-10){(b)}
 \end{overpic}
\end{center}
\end{minipage}
\end{tabular}
\end{center}
\caption{}
\label{ex1}
\end{figure}

\begin{figure}[htbp]
\begin{center}
\begin{tabular}{c}
\begin{minipage}{0.4\hsize}
\begin{center}
 \begin{overpic}[width=40mm]{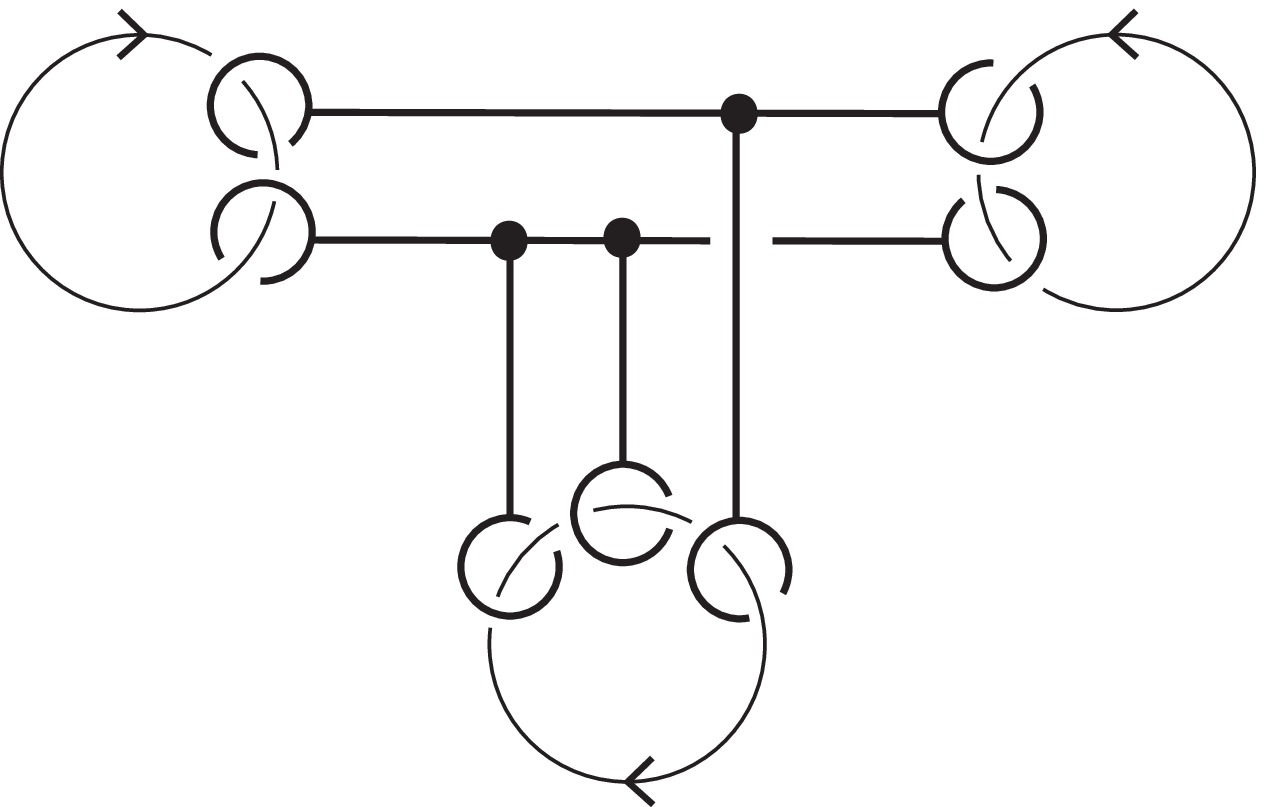}
   \linethickness{0.5pt}
   \put(145,31){\vector(-1,0){20}}
   \put(60,70){$T_{1}$}
   \put(33,35){$T_{2}$}
   \put(-14,51){$O_{1}$}
   \put(117,51){$O_{2}$}
   \put(73,7){$O_{3}$}
   \put(52,-10){(a)}
 \end{overpic}
\end{center}
\end{minipage}

\begin{minipage}{0.6\hsize}
\vspace{8.5mm}\begin{center}
 \begin{overpic}[width=56mm]{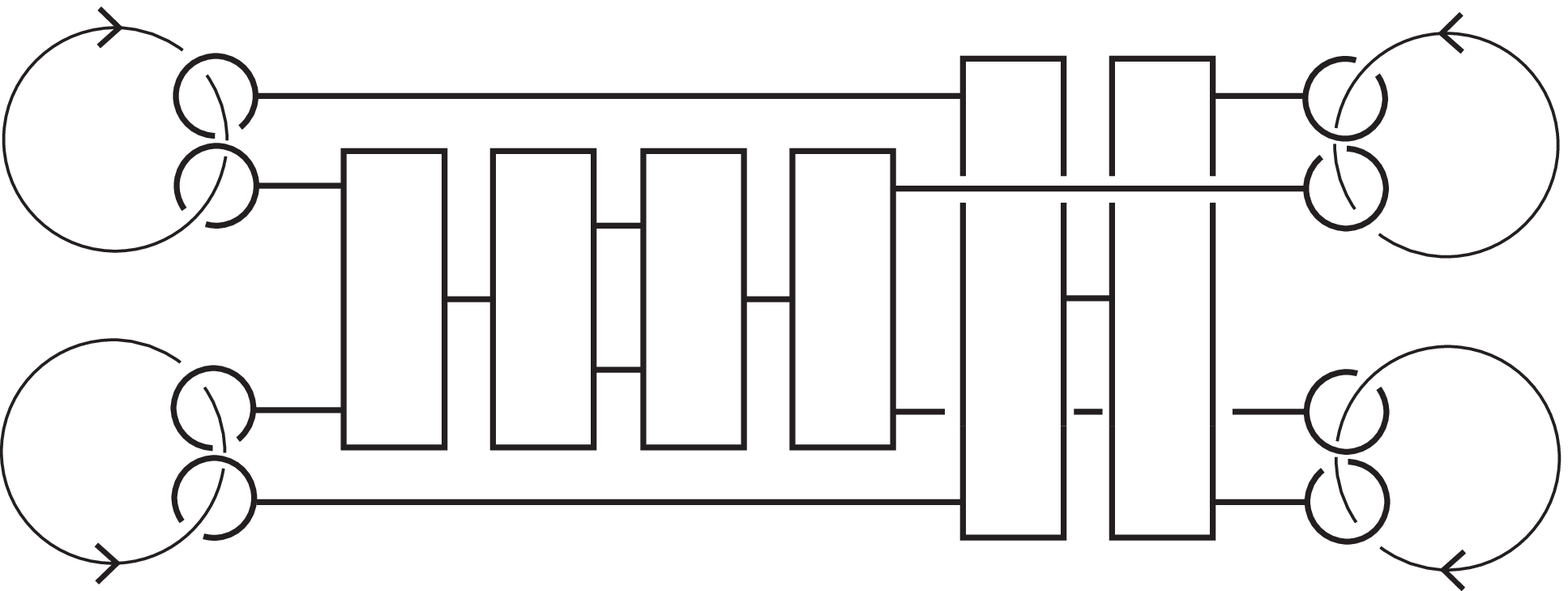}
  \put(-14,49){$O_{1}^{0}$}
  \put(-14,2){$O_{1}^{1}$}
  \put(161,49){$O_{2}^{0}$}
  \put(161,2){$O_{2}^{1}$}
  \put(74,-10){(b)}
 \end{overpic}
\end{center}
\end{minipage}
\end{tabular}
\end{center}
\caption{}
\label{ex2}
\end{figure}

\begin{example}
Let $O=O_{1}\cup O_{2}\cup O_{3}$ be an oriented $3$-component trivial link in $S^{3}$.
Let $L$ be a link which is obtained from $O$ by surgery along a linear $C_{2}^{d}$-tree $T_{1}$ as illustrated in Figure~\ref{ex1} (a),
and $L'$ a link which is obtained from $O$ by surgery along a disjoint union of $T_{1}$ and a simple linear $C_{3}$-tree $T_{2}$ as illustrated in Figure~\ref{ex2} (a). 
It is not hard to see that  $L$ and $L'$ are link-homotopic.

Since the both $L$ and $L'$ are Brunnian links,  
we have  the double branched cover of $S^{3}$ branched over $O_{3}$ which is a component of $L$ (resp. $L'$). 
Moreover we have surgery descriptions  as illustrated in Figure~\ref{ex1} (b) and Figure~\ref{ex2} (b).
Then covering links $L(00)$ and $L(01)$ of $L$ are links as illustrated in Figure~\ref{ex12}, and hence 
we have $\omu_{L(00)}(12)=1, \omu_{L(01)}(12)=-1$.
We conclude that $M_{L}(12)=\{1,-1\}$.
On the other hand,  covering links $L'(00)$ and $L'(01)$ of $L'$ are links as illustrated in Figure~\ref{ex22}.
Since $\omu_{L'(00)}(12)=3, \omu_{L'(01)}(12)=-3$,
we have that $M_{L'}(12)=\{3,-3\}$.
Therefore $L$ and $L'$ are link-homotopic, and $M_{L}(12)\neq M_{L'}(12)$.
\end{example}

\begin{figure}[htbp]
\begin{tabular}{lr}
\begin{minipage}{0.45\hsize}
\begin{center}
 \begin{overpic}[width=60mm]{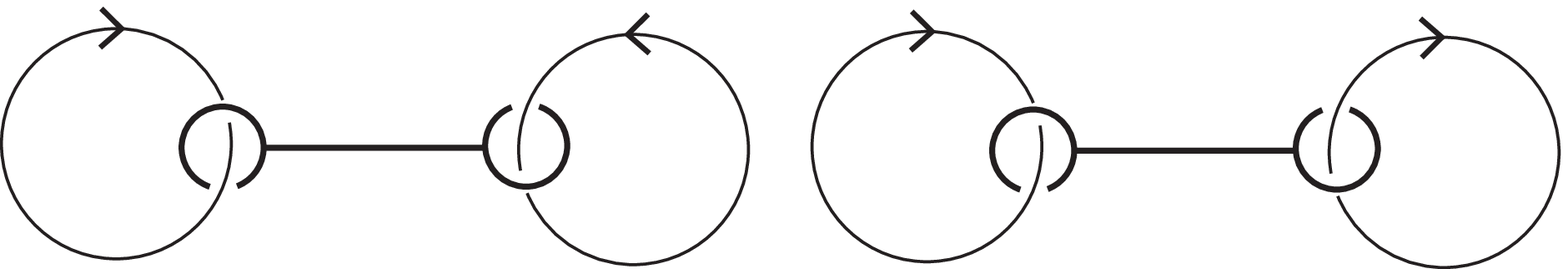}
  \put(30,-10){$L(00)$}
  \put(118,-10){$L(01)$}
 \end{overpic}
\caption{Covering~links~of~$L$}
\label{ex12}
\end{center}
\end{minipage}
\begin{minipage}{0.55\hsize}
\begin{center}
 \begin{overpic}[width=60mm]{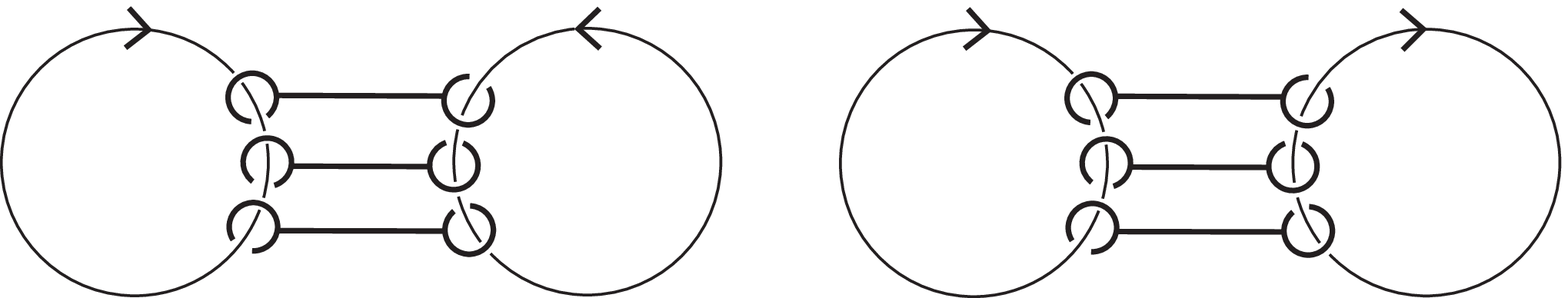}
  \put(30,-10){$L'(00)$}
  \put(119,-10){$L'(01)$}
 \end{overpic}
\caption{Covering~links~of~$L'$}
\label{ex22}
\end{center}
\end{minipage}
\end{tabular}
\end{figure}

\begin{remark}
\label{remstrong}
Since $L$ and $L'$ are Brunnian links and they are link-homotopic, 
$\omu_{L}(I)=\omu_{L'}(I)$ for any sequence $I$ with the length at most $3$. 
We note that $L$ is the Borromean rings. Hence 
$|\omu_L(123)|=|\omu_L(132)|=1$. 
It follows that  
for any sequence $I$ with the length at least $4$, 
$\Delta_{L}(I)_0=\Delta_{L'}(I)_0=\mathbb{Z}$ if each $i\in\{1,2,3\}$ appears in $I$, or 
$\mu_{L}(I)=\mu_{L'}(I)=0$ otherwise. 
In both cases, we have $\omu_{L}(I)=\omu_{L'}(I)=0$. 
This implies that $L$ and $L'$ have the same ordinary Milnor invariants,
and different covering Milnor invariants.
\end{remark}


\end{document}